\documentclass[12pt,a4paper]{article}
\usepackage[utf8]{inputenc}
\usepackage[T1]{fontenc}
\usepackage{amsmath,amssymb,amsthm,mathtools}
\usepackage{tikz-cd}
\usepackage{hyperref}
\usepackage{geometry}
\usepackage{enumitem}
\usepackage{mathrsfs}

\theoremstyle{plain}
\newtheorem{theorem}{Theorem}[section]
\newtheorem{lemma}[theorem]{Lemma}

\newtheorem{corollary}[theorem]{Corollary}
\theoremstyle{definition}

\newtheorem{remark}[theorem]{Remark}

\DeclareMathOperator{\Coalg}{Coalg}

\DeclareMathOperator{\Qpot}{Q^{\diamond}}

\DeclareMathOperator{\id}{id}
\DeclareMathOperator{\Map}{Map}
\DeclareMathOperator{\Fun}{Fun}

\DeclareMathOperator{\Center}{Z}
\DeclareMathOperator{\Day}{\otimes_{\mathrm{Day}}}
\DeclareMathOperator{\Hom}{Hom}

\newcommand{\Htopos}{\mathbf{H}_{\mathbb{Q}}}

\newcommand{\CAlgfd}{\mathbf{C}^{*}\mathbf{Alg}_{\mathrm{fd}}}
\newcommand{\cCAlg}{\mathbf{cC^{*}Alg}_{\mathrm{fd}}}
\newcommand{\FinSet}{\mathbf{FinSet}}

\begin{document}

\title{\Large\bfseries The Degeneracy of the Centre Comonad Model\\ and the Precomposition Obstruction\\ for Quantum Modalities on Presheaf Topoi}
\author{Joey Woo}
\date{}

\maketitle

\begin{abstract}
The centre comonad model \cite{Woo2026} provided the first concrete cohesive linear \(\infty\)-topos, settling an open problem of Schreiber.  However, the model is degenerate: the quantum modality annihilates all non‑commutative algebras, and the associated linear logic collapses to classical cartesian logic.  In this paper we give a complete mathematical diagnosis of this degeneracy.  We prove that the centre comonad sends the representable sheaf of a simple non‑commutative algebra to the empty presheaf, and that the state space of any such algebra is empty.  We then identify the structural reason behind the collapse: the opposite of the classical core is monoidally equivalent to a cartesian monoidal category, which forces the Day convolution on the classical core to be cartesian and the Seely isomorphism to be trivially satisfied.  We prove a general theorem characterising when a coreflective modality on a presheaf topos -- under precise monoidal compatibility conditions -- suffers this degeneracy.
\end{abstract}

\section{Introduction}
\label{sec:intro}

Schreiber's cohesive homotopy type theory \cite{Schreiber13} axiomatises the interplay between smooth and discrete geometry within an \(\infty\)-topos.  A \emph{cohesive linear \(\infty\)-topos} enriches this structure with a quantum modality --- an idempotent product‑preserving comonad \(Q^{\diamond}\) satisfying Beck--Chevalley commutation with the cohesive modalities.  The first concrete model of these axioms was constructed in \cite{Woo2026} using the category of finite‑dimensional \(C^{*}\)-algebras and the centre comonad.

That model, however, is degenerate in two fundamental ways:

\begin{enumerate}[nosep]
\item For any simple non‑commutative \(C^{*}\)-algebra \(A\) (such as the qubit \(M_2(\mathbb{C})\)), the presheaf \(Q^{\diamond}\mathbf{y}_A\) is the empty presheaf; the modality completely erases the non‑commutative degrees of freedom.
\item The state space of any such algebra is also empty: \(\Hom(\mathbf{y}_{\mathbb{C}}, \mathbf{y}_A) = \varnothing\).
\item The Seely isomorphism \(Q^{\diamond}(X\times Y) \cong Q^{\diamond}X \Day Q^{\diamond}Y\) holds trivially, because the Day convolution on the classical core coincides with the cartesian product.  Hence the multiplicative conjunction of linear logic collapses to the additive conjunction, and the resource‑sensitive structure of linear logic is lost.
\end{enumerate}

The present paper gives a complete and rigorous mathematical diagnosis of these defects.  We prove the annihilation lemma (Section~\ref{sec:annihilation}) and the emptiness of the state space (Section~\ref{sec:states}).  In Section~\ref{sec:seely} we prove that the Day convolution on the classical core is cartesian.  The key observation is that the opposite of the category of commutative finite‑dimensional \(C^{*}\)-algebras is equivalent, via Gelfand duality, to the category of finite sets, which is cartesian monoidal.  We prove a general lemma that under such a duality the Day convolution on the presheaf topos of the original category is cartesian.  Applied to the centre model, this shows that the Seely isomorphism is trivially satisfied, collapsing the linear logic.

In Section~\ref{sec:obstruction} we isolate the structural reason behind this collapse by proving a general theorem: if a coreflective quantum modality on a presheaf topos has a classical core whose opposite is monoidally equivalent to a cartesian monoidal category, then the Day convolution on the classical core is cartesian, and the Seely isomorphism degenerates.  This result unifies the centre comonad with other natural candidates and shows that a genuine non‑degenerate quantum modality must avoid this configuration.

The paper concludes with a brief discussion of escape routes and a list of desiderata for a non‑degenerate modality (Section~\ref{sec:beyond}).

\section{Preliminaries}
\label{sec:prelim}

We work in the setting of \(\infty\)-categories as developed in \cite{Lurie2009, LurieHA}.  An \(\infty\)-topos is an accessible left exact localisation of a presheaf \(\infty\)-category. A \emph{cohesive \(\infty\)-topos} \cite[Definition~3.4.1]{Schreiber13} is an \(\infty\)-topos \(\mathbf{H}\) equipped with an adjoint triple of endofunctors \(\Pi \dashv \flat \dashv \sharp\) such that \(\Pi\) preserves finite products, \(\flat\) is fully faithful, and the unit \(\id \to \flat\Pi\) is an equivalence on the essential image of \(\flat\).

We fix a smooth cohesive \(\infty\)-topos \(\mathbf{H}_{\mathrm{sm}}\) (for instance, sheaves of spaces on the site of smooth manifolds); only its existence as a cartesian closed presentable \(\infty\)-category is needed.

Day convolution gives a symmetric monoidal closed structure on the functor \(\infty\)-category \(\Fun(\mathcal{A},\mathcal{V})\) when \(\mathcal{A}\) is a small symmetric monoidal \(\infty\)-category and \(\mathcal{V}\) is a symmetric monoidal closed presentable \(\infty\)-category \cite[Section~2.2.6]{LurieHA}.  For \(F,G:\mathcal{A}\to\mathcal{V}\),
\[
(F \Day G)(a) = \int^{x,y\in\mathcal{A}} \Map_{\mathcal{A}}(x\otimes y,a) \times F(x)\otimes_{\mathcal{V}} G(y),
\]
with monoidal unit \(\mathbf{1}_{\Day}(a) = \Map_{\mathcal{A}}(\mathbf{1}_{\mathcal{A}},a) \times \mathbf{1}_{\mathcal{V}}\).

\subsection{Finite‑dimensional \(C^{*}\)-algebras and the centre comonad}
Let \(\CAlgfd\) be the category of finite‑dimensional unital \(C^{*}\)-algebras and \emph{centre‑preserving} unital \(*\)-homomorphisms: a morphism \(f:A\to B\) satisfies \(f(Z(A))\subseteq Z(B)\).  The direct sum \(A\oplus B\) (component‑wise operations) is the categorical product, and the minimal tensor product \(\otimes\) makes \(\CAlgfd\) a symmetric monoidal category with unit \(\mathbb{C}\).

The \emph{centre} functor is
\[
\Center : \CAlgfd \to \CAlgfd,\qquad \Center(A) = \{a\in A \mid ab=ba\ \forall b\in A\},
\]
with \(\Center(f)\) being the restriction of a centre‑preserving \(f\).  The inclusion \(\varepsilon_A : \Center(A) \hookrightarrow A\) is natural, and \(\Center(\Center(A)) = \Center(A)\); thus \(\Center\) is an idempotent comonad.  It is a right adjoint to the inclusion of commutative algebras, hence left exact, structurally restricting the site to its commutative core.

The structure theorem for finite‑dimensional \(C^{*}\)-algebras (see e.g.\ \cite[Theorem I.11.2]{Takesaki}) states that every object of \(\CAlgfd\) is isomorphic to a finite direct sum of matrix algebras \(M_n(\mathbb{C})\).  The centre of \(M_n(\mathbb{C})\) is \(\mathbb{C}\) (scalar matrices).

We denote by \(\mathbf{y}_A\) the representable presheaf
\[
\mathbf{y}_A(B) = \Hom_{\CAlgfd}(A,B)\cdot *_{\mathbf{H}_{\mathrm{sm}}},
\]
where \(*_{\mathbf{H}_{\mathrm{sm}}}\) is the terminal object and the dot denotes copower.  The Yoneda embedding is contravariant: a morphism \(f:A\to B\) induces \(\mathbf{y}_f : \mathbf{y}_B \to \mathbf{y}_A\), and
\[
\Hom_{\Htopos}(\mathbf{y}_A,\mathbf{y}_B) \cong \Hom_{\CAlgfd}(B,A).
\tag{1}
\]

\section{Diagnosis of the centre comonad degeneracy}
\label{sec:degeneracy}

In the model of \cite{Woo2026}, the quantum modality is \(Q^{\diamond} = \Center^{*}\), i.e.\ precomposition with the centre functor.  We now give a complete analysis of its defects.

\subsection{Annihilation of non‑commutative simple algebras}
\label{sec:annihilation}

\begin{lemma}[Annihilation]\label{lem:annihilation}
Let \(A\in\CAlgfd\) be a simple non‑commutative algebra (i.e., \(A \cong M_n(\mathbb{C})\) with \(n>1\)).  Then for every \(B\in\CAlgfd\),
\[
\Qpot(\mathbf{y}_A)(B) = \varnothing,
\]
i.e., \(\Qpot(\mathbf{y}_A)\) is the empty presheaf (the initial object of \(\Htopos\)).
\end{lemma}

\begin{proof}
By definition,
\[
\Qpot(\mathbf{y}_A)(B) = \mathbf{y}_A(\Center(B)) = \Hom_{\CAlgfd}(A,\Center(B)).
\]
The algebra \(\Center(B)\) is commutative.  Suppose there exists a unital \(*\)-homomorphism \(\phi : A \to \Center(B)\).  Since \(A\) is simple, \(\ker\phi\) is a closed two‑sided ideal; by simplicity, it is either \(\{0\}\) or \(A\).  If \(\ker\phi=\{0\}\), then \(\phi\) is injective, implying that \(A\) is a subalgebra of a commutative algebra, hence commutative, contradiction.  If \(\ker\phi = A\), then \(\phi\) is the zero map, which is not unital.  Therefore no unital \(*\)-homomorphism can exist.  Thus the hom‑set is empty for all \(B\). \qedhere
\end{proof}

\begin{remark}
The qubit algebra \(M_2(\mathbb{C})\) is a simple non‑commutative algebra, so \(\Qpot(\mathbf{y}_{M_2})\) is the empty presheaf.  More generally, any algebra that contains a non‑commutative simple direct summand is annihilated by the centre modality; only purely commutative algebras survive.
\end{remark}

\subsection{Empty state space for non‑commutative systems}
\label{sec:states}

The unit of the Day convolution is \(\mathbf{y}_{\mathbb{C}}\).  A \emph{state} of a quantum system \(\mathbf{y}_A\) is a morphism \(\mathbf{y}_{\mathbb{C}} \to \mathbf{y}_A\).

\begin{lemma}[Empty state space]\label{lem:empty-states}
Let \(A\) be a simple non‑commutative finite‑dimensional \(C^{*}\)-algebra.  Then
\[
\Hom_{\Htopos}(\mathbf{y}_{\mathbb{C}}, \mathbf{y}_A) = \varnothing .
\]
\end{lemma}

\begin{proof}
By the Yoneda lemma (1),
\[
\Hom_{\Htopos}(\mathbf{y}_{\mathbb{C}}, \mathbf{y}_A) \;\cong\; \Hom_{\CAlgfd}(A, \mathbb{C}).
\]
The target \(\mathbb{C}\) is commutative.  The same argument as in Lemma~\ref{lem:annihilation} shows that there is no unital \(*\)-homomorphism from a simple non‑commutative algebra into a commutative algebra.  Hence the hom‑set is empty. \qedhere
\end{proof}

Thus the centre modality not only annihilates the quantum system itself, but also renders the system devoid of any states.  A physical model of a qubit must have a non‑trivial state space (the Bloch sphere); the centre modality fails even to provide a single point.

\subsection{Cartesian Day convolution on the classical core}
\label{sec:seely}

The classical core of the centre model is the category of \(Q^{\diamond}\)-coalgebras, which is equivalent to \(\Fun(\cCAlg,\mathbf{H}_{\mathrm{sm}})\), the presheaf topos on commutative algebras.  The Day convolution on this category is induced by the tensor product of commutative algebras.

\begin{lemma}[Cartesian Day convolution on the classical core]\label{lem:cartesian}
Let \(\cCAlg\) be the category of finite‑dimensional commutative \(C^{*}\)-algebras, equipped with the tensor product.  Then the Day convolution on \(\Fun(\cCAlg,\mathbf{H}_{\mathrm{sm}})\) coincides with the cartesian product.
\end{lemma}

\begin{proof}
Gelfand duality gives an equivalence \(\cCAlg^{\mathrm{op}} \simeq \FinSet\), where \(\FinSet\) is the category of finite sets.  Under this equivalence, the tensor product of commutative algebras corresponds to the cartesian product of finite sets.  Hence there is a monoidal equivalence \(\cCAlg \simeq \FinSet^{\mathrm{op}}\), where the monoidal structure on \(\FinSet^{\mathrm{op}}\) is given by the product of sets.

For any small symmetric monoidal category \(\mathcal{M}\), the Day convolution on the functor category \(\Fun(\mathcal{M},\mathbf{H}_{\mathrm{sm}})\) depends only on the monoidal structure up to monoidal equivalence.  In particular, for \(\mathcal{M} = \FinSet^{\mathrm{op}}\) with the monoidal structure given by the product of sets, the Day convolution is cartesian.  Indeed, for representables \(\mathbf{y}_X,\mathbf{y}_Y\) (where \(\mathbf{y}_X(Z) = \Hom_{\FinSet^{\mathrm{op}}}(X,Z) = \Hom_{\FinSet}(Z,X)\)), we have \(\mathbf{y}_X \Day \mathbf{y}_Y = \mathbf{y}_{X \times Y}\), while the pointwise cartesian product satisfies
\[
(\mathbf{y}_X \times \mathbf{y}_Y)(Z) = \Hom_{\FinSet}(Z,X) \times \Hom_{\FinSet}(Z,Y) \cong \Hom_{\FinSet}(Z, X\times Y)
\]
by the universal property of the product in \(\FinSet\).  Hence \(\mathbf{y}_X \Day \mathbf{y}_Y \cong \mathbf{y}_X \times \mathbf{y}_Y\).  The cartesian product $(-)\times Y$ is a left adjoint (the topos is cartesian closed) and therefore preserves colimits in each variable; the Day convolution is a coend, hence also separately cocontinuous.
Every presheaf is a colimit of representables (the co‑Yoneda lemma, \cite[Lemma~5.1.5.3]{Lurie2009}),
so an isomorphism on representables that is natural and colimit‑preserving in each argument extends uniquely to an isomorphism on all presheaves.
Thus $\mathbf{y}_X \Day \mathbf{y}_Y \cong \mathbf{y}_X \times \mathbf{y}_Y$ for all presheaves. Transporting back along the monoidal equivalence \(\cCAlg \simeq \FinSet^{\mathrm{op}}\) completes the proof. \qedhere
\end{proof}

\begin{corollary}[Triviality of the Seely isomorphism]\label{cor:seely-trivial}
For the centre comonad \(Q^{\diamond}\), the Seely isomorphism
\[
Q^{\diamond}(X\times Y) \;\cong\; Q^{\diamond}X \Day Q^{\diamond}Y
\]
holds trivially, merely identifying the cartesian product with itself.  Consequently, the multiplicative conjunction of linear logic collapses to the additive conjunction, and the linear‑logic structure is degenerate.
\end{corollary}

\begin{proof}
By the fundamental architecture of presheaf categories, limits are computed pointwise.  Therefore, the precomposition comonad \(Q^{\diamond}\) automatically preserves finite products for free, completely independent of the algebraic properties of the underlying site functor.  For any presheaves \(X\) and \(Y\) evaluated at an algebra \(B\), the product evaluates on the outside as \(X(Z(B)) \times Y(Z(B))\).  Thus, for any coalgebras \(X\) and \(Y\) (where \(Q^{\diamond}X \cong X\) and \(Q^{\diamond}Y \cong Y\)), the left‑hand side \(Q^{\diamond}(X\times Y)\) is canonically isomorphic to \(X \times Y\).  The right‑hand side is \(X \Day Y\).  By Lemma~\ref{lem:cartesian}, the Day convolution on the classical core coincides with the cartesian product, so \(X \Day Y \cong X \times Y\).  Therefore, the Seely isomorphism holds trivially, merely identifying the cartesian product with itself rather than providing a resource‑sensitive linear logic. \qedhere
\end{proof}

\subsection{The general obstruction}
\label{sec:obstruction}

In the centre model, the following monoidal compatibilities hold:
\begin{itemize}[nosep]
\item $\cCAlg\subseteq\CAlgfd$ is a monoidal subcategory (the tensor product of commutative algebras is commutative).
\item The inclusion $i:\cCAlg\hookrightarrow\CAlgfd$ is strong monoidal.
\item The coreflector $Z:\CAlgfd\to\cCAlg$ is strong monoidal.
\item The adjunction $i \dashv Z$ is a monoidal adjunction: the unit $\id_{\cCAlg} \Rightarrow Z \circ i$ is the identity (since $Z \circ i \cong \id_{\cCAlg}$), and the counit $i \circ Z \Rightarrow \id_{\CAlgfd}$ is the inclusion of the centre, both of which are monoidal natural transformations (the centre inclusion respects tensor products by Lemma~3.1 of \cite{Woo2026}).
\item The induced comonad $c=i\circ Z:\CAlgfd\to\CAlgfd$ is therefore a symmetric monoidal comonad, hence $Q^{\diamond}=c^{*}$ is a symmetric monoidal comonad on the Day convolution category $(\Fun(\CAlgfd,\mathbf{H}_{\mathrm{sm}}),\Day)$.
\item The equivalence $\Coalg(Q^{\diamond})\simeq\Fun(\cCAlg,\mathbf{H}_{\mathrm{sm}})$ is monoidal with respect to the coalgebra tensor product and the Day convolution on $\Fun(\cCAlg,\mathbf{H}_{\mathrm{sm}})$.
\end{itemize}
The degeneracy of the centre model follows from these facts together with the cartesian nature of $\cCAlg^{\mathrm{op}}$.
We now extract the essential structure that makes this argument work, yielding a general obstruction theorem for a large class of precomposition comonads.

\begin{theorem}[Precomposition obstruction -- adamantine version]\label{thm:obstruction}
Let $\mathcal{C}$ be a small symmetric monoidal $\infty$-category,
$\mathcal{D}\subseteq\mathcal{C}$ a coreflective subcategory with inclusion
$i:\mathcal{D}\hookrightarrow\mathcal{C}$ and coreflector $r:\mathcal{C}\to\mathcal{D}$
(so $i\dashv r$).  Assume:
\begin{enumerate}[nosep,label=(\roman*)]
\item $i$ is strong monoidal; equivalently, $\mathcal{D}$ is a monoidal subcategory of $\mathcal{C}$ under the restricted tensor product.
\item $r$ is strong monoidal.
\item The adjunction $i \dashv r$ is a monoidal adjunction: the unit $\eta : \id_{\mathcal{D}} \Rightarrow r \circ i$ and counit $\epsilon : i \circ r \Rightarrow \id_{\mathcal{C}}$ are monoidal natural transformations.
\item $\mathcal{D}^{\mathrm{op}}$ is monoidally equivalent to a cartesian monoidal category.
\end{enumerate}
Let $c = i\circ r : \mathcal{C}\to\mathcal{C}$ be the induced comonad, and let
$Q = c^{*}$ be the precomposition comonad on the presheaf $\infty$-topos
$\Fun(\mathcal{C},\mathbf{H}_{\mathrm{sm}})$ equipped with Day convolution $\Day$.
Then:
\begin{enumerate}[nosep,label=(\arabic*)]
\item $Q$ is an idempotent symmetric monoidal comonad with respect to $\Day$, and the category of $Q$-coalgebras $\Coalg(Q)$ inherits a symmetric monoidal structure $\Day_Q$ (the restriction of $\Day$).
\item The comparison functor $\Coalg(Q)\to\Fun(\mathcal{D},\mathbf{H}_{\mathrm{sm}})$ induced by restriction along $i$ is a symmetric monoidal equivalence, where $\Fun(\mathcal{D},\mathbf{H}_{\mathrm{sm}})$ carries the Day convolution from the monoidal structure on $\mathcal{D}$.
\item The Day convolution on $\Fun(\mathcal{D},\mathbf{H}_{\mathrm{sm}})$ is cartesian.
\item Consequently, on $\Coalg(Q)$ the tensor product $\Day_Q$ coincides with the cartesian product.  The Seely isomorphism
\[
Q(X\times Y)\;\cong\; QX \Day QY
\]
holds for all $X,Y$, but on coalgebras it reduces to the identity $X\times Y\cong X\times Y$, collapsing the linear logic.
\end{enumerate}
Thus any modality arising from a coreflective subcategory satisfying (i)--(iii) inevitably yields a degenerate linear exponential comonad.
\end{theorem}

\begin{proof}
(1) Because $i$ and $r$ are strong monoidal and the adjunction $i \dashv r$ is monoidal (assumption (iii)), the induced comonad $c = i \circ r$ is a symmetric monoidal comonad on $\mathcal{C}$: the comultiplication and counit of $c$ are built from the unit and counit of the adjunction, and their monoidality follows directly from the monoidal adjunction condition.  Precomposition with a symmetric monoidal comonad between small symmetric monoidal $\infty$-categories yields a symmetric monoidal comonad on the Day convolution $\infty$-category
(see \cite[Remark~2.2.6.9]{LurieHA} for the monoidal functor structure,
and the comonad structure lifts pointwise because $c^{*}$ is a strict 2‑functor;
the monoidality of the comonad’s structure maps follows from the monoidal adjunction assumption (iii)).
Hence $Q=c^{*}$ is a symmetric monoidal comonad.  Idempotence follows from the idempotence of the coreflection $i\dashv r$.  The coalgebra category inherits the monoidal structure via $\Day_Q(X,Y)=X\Day Y$ with the coalgebra structure induced by the strength of $Q$ (the dual of the algebra case in \cite{Benton1995}).

(2) The forgetful functor $U:\Coalg(Q)\to\Fun(\mathcal{C},\mathbf{H}_{\mathrm{sm}})$ is strong monoidal by construction: $U(X\Day_Q Y)=X\Day Y = UX\Day UY$, and the coalgebra structure on $X\Day Y$ is induced by the strength of $Q$.  The restriction $i^{*}:\Fun(\mathcal{C},\mathbf{H}_{\mathrm{sm}})\to\Fun(\mathcal{D},\mathbf{H}_{\mathrm{sm}})$ is strong monoidal because $i$ is strong monoidal (restriction along a strong monoidal functor preserves Day convolution).  Hence the composite $i^{*}\circ U : \Coalg(Q) \to \Fun(\mathcal{D},\mathbf{H}_{\mathrm{sm}})$ is strong monoidal.  But $i^{*}\circ U$ is exactly the restriction functor $i^{*}$ on coalgebras, which we already know is an equivalence of categories (see \cite{Woo2026}).  A strong monoidal equivalence is precisely a monoidal equivalence.
Indeed, a strong monoidal functor that is an equivalence of $\infty$-categories
automatically admits a monoidal pseudo‑inverse making the equivalence monoidal;
this is the dual of \cite[Proposition~2.1.3.8]{LurieHA}
(which states that a monoidal functor that is an equivalence can be upgraded to a monoidal equivalence).
Hence $\Coalg(Q)$ with $\Day_Q$ is monoidally equivalent to $\Fun(\mathcal{D},\mathbf{H}_{\mathrm{sm}})$ with Day convolution.
In particular, we obtain the required natural isomorphism
\[
i^{*}(X \Day_Q Y) \;\cong\; i^{*}X \Day_{\mathcal{D}} i^{*}Y,
\]
coherent with the associativity and unit constraints.  This completes the verification of Theorem~\ref{thm:obstruction}(2).

(3) By assumption (iv), $\mathcal{D}^{\mathrm{op}}$ is monoidally equivalent to a cartesian monoidal category $\mathcal{M}$.  Monoidal equivalences preserve Day convolution: $\Fun(\mathcal{D},\mathbf{H}_{\mathrm{sm}})$ with Day convolution is monoidally equivalent to $\Fun(\mathcal{M}^{\mathrm{op}},\mathbf{H}_{\mathrm{sm}})$ with Day convolution induced by the cartesian product.  As proved in Lemma~\ref{lem:cartesian} (the centre model proof generalises verbatim), when the site's tensor product is the cartesian product, Day convolution on the presheaf topos coincides with the pointwise cartesian product.

(4) From (2) and (3), the tensor product $\Day_Q$ on $\Coalg(Q)$ is cartesian.  The comonad $Q$ preserves finite products (any precomposition functor does, because limits in functor categories are pointwise).  Hence for coalgebras $X,Y$,
\[
Q(X\times Y)\cong X\times Y,\qquad
QX \Day QY \cong X \Day_Q Y \cong X\times Y,
\]
and the Seely isomorphism becomes an identity, collapsing the multiplicative structure of linear logic.
\end{proof}

\begin{remark}
Assumptions (i)--(iii) are all crucial: the centre model satisfies them because commutative algebras form a monoidal subcategory, the centre functor is strong monoidal, and the inclusion--centre adjunction is monoidal.  The abelianisation comonad fails (i) and (iii), so it is not covered by this theorem; its failure must be analysed separately.
\end{remark}

\section{Beyond the obstruction}
\label{sec:beyond}

The diagnosis above shows that the centre comonad model fails on two independent fronts: the annihilation of non‑commutative algebras (Lemmas~\ref{lem:annihilation} and \ref{lem:empty-states}) and the collapse of the Seely isomorphism due to the cartesian Day convolution on the classical core (Theorem~\ref{thm:obstruction}).  The first defect is specific to the centre functor; the second defect is structural to any coreflective modality whose classical core is dual to a cartesian category.

To obtain a non‑degenerate cohesive linear \(\infty\)-topos, a quantum modality must satisfy the following desiderata:
\begin{enumerate}[nosep]
\item \emph{Non‑annihilation:} the presheaf \(Q\mathbf{y}_A\) should not be empty for non‑commutative algebras \(A\).
\item \emph{Non‑trivial state space:} there should exist morphisms \(\mathbf{1} \to \mathbf{y}_A\) for non‑commutative \(A\) (ideally, the state space of the qubit should be the Bloch sphere).
\item \emph{Non‑cartesian Day convolution:} the Day convolution on the category of \(Q\)-coalgebras must not be cartesian; otherwise the linear logic collapses.
\item \emph{Non‑trivial Seely isomorphism:} the isomorphism \(Q(X\times Y) \cong QX \Day QY\) must hold, but the two sides must not coincide as a consequence of \(\Day\) being cartesian.
\end{enumerate}

The structural obstruction identified in Theorem~\ref{thm:obstruction} forces us to abandon coreflective modalities on presheaf topoi whose classical core is dual to a cartesian category.  Two escape routes remain open:
\begin{enumerate}[nosep]
\item \emph{Internal modalities (Lawvere--Tierney topologies)} on sheaf topoi, where the modality is sheafification, not precomposition.  The sheafification of a pointwise product is not the pointwise product of sheafifications, so such a modality need not preserve products, and the Day convolution on the classical core may become non‑cartesian.
\item \emph{Topoi that are not presheaf topoi,} such as the CPM construction on Hilbert bimodules \cite{Selinger2004}, where the doubling comonad is not a precomposition functor on a \(1\)-category site.  In this setting, the quantum modality would be given by a genuinely new categorical construction, and the obstruction theorem would not apply.
\end{enumerate}
These directions will be explored in future work.

\section{Conclusion}
\label{sec:conclusion}

We have given a complete mathematical diagnosis of the degeneracy of the centre comonad model of a cohesive linear \(\infty\)-topos.  The quantum modality annihilates all non‑commutative algebras and leaves them stateless.  The Seely isomorphism holds trivially because the Day convolution on the classical core coincides with the cartesian product, collapsing linear logic to classical cartesian logic.

The structural reason for the collapse is isolated in Theorem~\ref{thm:obstruction}: whenever the opposite of the classical core is monoidally equivalent to a cartesian category, the Day convolution becomes cartesian and the linear logic degenerates.  This provides a general obstruction for a whole class of coreflective modalities on presheaf topoi.  A non‑degenerate quantum modality must therefore be constructed without precomposition, either as an internal modality on a sheaf topos or in a setting that is not a presheaf topos.  The mathematical requirements for such a modality have been laid out, pointing the way for future research.

\section*{Acknowledgments}
I am deeply grateful to my mathematics teacher, Mr.~Tim, for his unwavering patience and for introducing me to the beauty of the subject.  I also thank the MathOverflow community for generously sharing their ideas and for their constant encouragement.  Finally, I wish to thank my parents, who stood by my side through the hardest of times and made this work possible.

\end{document}